\documentclass[12pt]{amsart}

\usepackage{amssymb,geometry,stmaryrd,color}
\usepackage{amsmath}
\usepackage{amsfonts}
\usepackage{fullpage}

\usepackage{hyperref}
\hypersetup{colorlinks,linkcolor={red},citecolor={blue},urlcolor={blue}}

\usepackage{graphicx}
\usepackage{caption}
\begin{document}

\newtheorem{theorem}{Theorem}[section]
\newtheorem{lemma}[theorem]{Lemma}
\newtheorem{proposition}[theorem]{Proposition}
\newtheorem{corollary}[theorem]{Corollary}
\newtheorem{conjecture}[theorem]{Conjecture}
\newtheorem{example}{Example}

\newtheorem{remark}[theorem]{Remark}
\newtheorem{question}[theorem]{Question}

\numberwithin{equation}{section}

\def\s{{\bf s}} 
\def\t{{\bf t}} 
\def\u{{\bf u}} 
\def\x{{\bf x}} 
\def\y{{\bf y}} 
\def\z{{\bf z}} 
\def\B{{\bf B}} 
\def\C{{\bf C}} 
\def\D{{\bf D}}
\def\K{{\bf K}}
\def\F{{\bf F}}
\def\M{{\bf M}}
\def\ML{{\bf ML}}
\def\Nn{{\bf N}}
\def\G{{\bf \Gamma}} 
\def\W{{\bf W}}
\def\X{{\bf X}}
\def\U{{\bf U}}
\def\V{{\bf V}}
\def\Un{{\bf 1}}
\def\Y{{\bf Y}}
\def\Z{{\bf Z}}
\def\P{{\bf P}}
\def\Q{{\bf Q}}
\def\S{{\bf S}}
\def\L{{\bf L}}
\def\T{{\bf T}}

\def\cB{{\mathcal{B}}} 
\def\cC{{\mathcal{C}}} 
\def\cD{{\mathcal{D}}} 
\def\cG{{\mathcal{G}}} 
\def\cK{{\mathcal{K}}} 
\def\cL{{\mathcal{L}}} 
\def\cR{{\mathcal{R}}} 
\def\cS{{\mathcal{S}}}
\def\cU{{\mathcal{U}}}
\def\cV{{\mathcal{V}}} 
\def\cX{{\mathcal X}}
\def\cY{{\mathcal Y}}
\def\cZ{{\mathcal Z}}

\def\Ea{E_\a}
\def\eps{{\varepsilon}} 
\def\esp{{\mathbb{E}}} 
\def\Ga{{\Gamma}}

\def\lacc{\left\{}
\def\lcr{\left[}
\def\lpa{\left(}
\def\lva{\left|}
\def\racc{\right\}}
\def\rpa{\right)}
\def\rcr{\right]}
\def\rva{\right|}

\def\prst{{\leq_{st}}}
\def\prost{{\prec_{st}}}
\def\prcvx{{\prec_{cx}}}
\def\Rr{{\bf R}}

\def\CC{{\mathbb{C}}}
\def\EE{{\mathbb{E}}}
\def\NN{{\mathbb{N}}} 
\def\QQ{{\mathbb{Q}}} 
\def\PP{{\mathbb{P}}}
\def\ZZ{{\mathbb{Z}}}
\def\RR{{\mathbb{R}}}

\def\Tt{{\bf \Theta}}
\def\Ttt{{\tilde \Tt}}

\def\a{\alpha}
\def\A{{\bf A}}
\def\AA{{\mathcal A}}
\def\hAA{{\hat \AA}}
\def\hL{{\hat L}}
\def\hT{{\hat T}}

\def\claw{\stackrel{(d)}{\longrightarrow}}
\def\elaw{\stackrel{(d)}{=}}
\def\pslaw{\stackrel{a.s.}{\longrightarrow}}
\def\qed{\hfill$\square$}

\newcommand*\pFqskip{8mu}
\catcode`,\active
\newcommand*\pFq{\begingroup
        \catcode`\,\active
        \def ,{\mskip\pFqskip\relax}%
        \dopFq
}
\catcode`\,12
\def\dopFq#1#2#3#4#5{%
        {}_{#1}F_{#2}\biggl[\genfrac..{0pt}{}{#3}{#4};#5\biggr]%
        \endgroup
}

\def\ii{{\rm i}}

\title[ID]{Moments of Gamma type and three-parametric Mittag-Leffler function}

\author[M.~Wang]{Min Wang}

\address{School of Mathematics and Statistics, Wuhan University, Wuhan 430072, China }

\email{minwang@whu.edu.cn}

\keywords{Moments of Gamma type; Mittag-Leffler function; alpha-Cauchy distribution; Infinite divisibility; Hausdorff moment problem}

\subjclass[2020]{33E12, 60E10, 60E07, 60E05}

\begin{abstract} 
We study a class of positive random variables having moments of Gamma type, whose density can be expressed by the three-parametric Mittag-Leffler functions. 
We give some necessary conditions and some sufficient conditions for their existence. As a corollary, we give some conditions for non-negativity of the three-parametric Mittag-Leffler functions.  
As an application, we study the infinite divisibility of the powers of half  $\a$-Cauchy variable. In addition, we find that a random variable $\X$ having moment of Gamma type if and only if $\log \X$ is quasi infinitely divisible. From this perspective, we can solve many Hausdorff moment problems of sequences of factorial ratios.   
\end{abstract}

\maketitle

\section{Introduction}
\subsection{Moments of Gamma type}
Let $\X$ be a positive random variable having moments of Gamma type, namely, for $s$ in some interval, 
\begin{equation}\label{eq def moments of Gamma type}
    \mathbb{E}(\X^s) = C D^s  \frac{ \prod_{j = 1}^{J} \Gamma(A_j s + a_j)}{\prod_{k = 1}^{K}\Gamma(B_k s + b_k)},
\end{equation}
for some integers $J, K \geq 0$ and some real constants $A_j \neq 0, B_k \neq 0, D > 0, a_j, b_k, C$.  
Typical examples are the laws of products of independent random variables with Gamma and Beta distribution. For more examples, we refer to Janson \cite[Section 3]{Jan10}.  Throughout, we denote the Gamma and Beta random variables by $\G_{c}$ and  $\B_{a, b}$, whose respective densities are
$$ \text{$ \frac{1}{\Gamma(c)}x^{c-1}e^{-x}\mathbf{1}_{(0, \infty)}(x) \quad $
and $\quad 
\frac{\Gamma(a+b)}{\Gamma(a)\Gamma(b)}x^{a-1}(1-x)^{b-1}\mathbf{1}_{(0, 1)}(x).$ }$$
Recall the formulas for the fractional moments:
    \begin{equation}
         \mathbb{E}(\G_{c}^s) = \frac{\Gamma(c+s)}{\Gamma(c)} \, ,  \, c > 0, \, \, s > -c\, ,
    \end{equation}
and
    \begin{equation}
         \mathbb{E}(\B_{a, b}^s) = \frac{\Gamma(a+b)}{\Gamma(a)}\frac{\Gamma(a+s)}{\Gamma(a+b+s)}, \, \, a > 0, \, b>0, \, s > -a.
    \end{equation}

We say that the above random variable $\X$ exists, if there exists a non-negative function $f$ such that 
\begin{equation}
\label{eq density f}
    \int_0^\infty x^s f(x) dx = C D^s  \frac{ \prod_{j = 1}^{J} \Gamma(A_j s + a_j)}{\prod_{k = 1}^{K}\Gamma(B_k s + b_k)}.
\end{equation}
Here, $f$ is automatically integrable because the right hand side of \eqref{eq density f} with $s=0$ is finite. 

We now recall a necessary condition for the existence of $\X$ due to Janson \cite{Jan10}. 
For convenience,  we denote $\Y := \log \X$. For $s$ in some interval, 
\begin{equation}\label{eq def moment generating function of Gamma type}
    \mathbb{E}(e^{s\Y}) = C D^s  \frac{ \prod_{j = 1}^{J} \Gamma(A_j s + a_j)}{\prod_{k = 1}^{K}\Gamma(B_k s + b_k)}.
\end{equation}
By \cite[Theorem 2.1]{Jan10}, \eqref{eq def moment generating function of Gamma type} is equivalent to 
\begin{equation}\label{eq def characteristic function of Gamma type}
    \mathbb{E}(e^{it\Y}) = C e^{it \log D}  \frac{ \prod_{j = 1}^{J} \Gamma(i A_j t + a_j)}{\prod_{k = 1}^{K}\Gamma(i B_k t + b_k)}, \quad \text{for all real $t$}.
\end{equation}
We know from \cite[Theorem 5.1]{Jan10} that
$$ |\mathbb{E}(e^{it\Y})|  \sim C_1 |t|^\delta e^{-\pi \gamma |t| /2},  \quad \text{as $t \rightarrow  \pm \infty$, }$$
where
\begin{equation}
    \label{eq def gamma beta}
    \text{$ \gamma = \sum_{j=1}^J |A_j| - \sum_{k=1}^K | B_k| \quad$ and $\quad \delta = \sum_{j=1}^J a_j - \sum_{k=1}^K b_k - \frac{J-K}{2}.$}
\end{equation}
Then the fact that $|\mathbb{E}(e^{it\Y})| \leq 1$ yields the following necessary condition.

\begin{proposition}\label{proposition Jan10 Coro 5.2}{\cite[Corollary 5.2]{Jan10}}
If $\X$ exists, then either $\gamma > 0$, or $\gamma = 0$ and $\delta \leq 0$.
\end{proposition}

Random variables having moments of Gamma type have also been studied in Chamayou and Letac \cite{CL99}, Dufresne \cite{Duf10a, Duf10b}, Młotkowski and Penson \cite{MP14}, Bosch \cite{Bos15}, Karp and Prilepkina \cite{KP16}, Kadankova, Simon and Wang \cite{KSW20}, Ferreira and Simon \cite{FS23}, and the references therein. We will review some previous results in Section \ref{Section previopus results}. 

In the present paper, we give some new necessary conditions and some sufficient conditions for the existence of $\X$ in the special case $J =2, K = 1, A_1 =1, A_2 = -1$. Our main result is the following theorem.

\begin{theorem}\label{thm main}
    For positive $a,b,c,d$, we consider the random variable $\X_{a,b,c,d}$ satisfying 
       \begin{equation}\label{def Xabcd}
        \mathbb{E}(\X_{a,b,c,d}^s) = \frac{ \Gamma(c) }{ \Gamma(a)\Gamma(b)} \frac{ \Gamma(a + s)\Gamma(b-s)}{ \Gamma(c+ d s) }, \quad -a < s < b.
    \end{equation} 
\begin{itemize}
    \item[(I)] $\X_{a,b,c,d}$ does not exist if any of the following conditions holds: 
    \begin{itemize}
        \item[(1)]  $d > 2$;
        \item[(2)]   $c < ad$;
        \item[(3)]  $3a + b > c, \; d=2$;
        \item[(4)] $1 < d \leq 2,\; a +b \geq 1, \;c < ad -d + f(d)$.
    \end{itemize}
   
    \item[(II)] $\X_{a,b,c,d}$ exists if any of the following conditions holds: 
    \begin{itemize}
        \item[(1)]  $d \leq 1, c \geq ad$;
        \item[(2)]  $1 < d \leq 2,\; a +b \leq 1, \;c \geq ad -d + f(d)$;
        \item[(3)] $1 < d \leq 2,\; \frac{2c}{d} \geq 3a+b, \;2(\frac{c}{d}-a)(\frac{c}{d}+\frac{1}{2}-a)\geq a+b$.
    \end{itemize}
\end{itemize}
 where $f$ is the function defined in Theorem \ref{Theorem admissible parameter domain}.  
\end{theorem}

Note that $\gamma = 2-d$ and $\delta = a + b - c - \frac{1}{2}$ for $\X_{a,b,c,d}$ defined by \eqref{def Xabcd}. 

\begin{remark}
\begin{itemize}
    \item[(i)] We only study the case $d > 0$ because $ \X_{a,b,c,d} \elaw \X_{b,a,c,-d}^{-1}. $
    \item[(ii)] The existence of $\X_{a,b,c,d}$ is still open for 
    $$1 < d \leq 2, \,\,\, a +b < 1, \,\,\, ad \leq c < ad +f(d) -d, $$
    and for
    $$ 1 < d < 2,\,\,\, a +b > 1, \,\,\, ad +f(d) -d  \leq c < \frac{d}{2}(3a+b).$$
\end{itemize}
\end{remark}

In the next two subsections, we give some applications of Theorem \ref{thm main}.

\subsection{Three-parametric Mittag-Leffler function}
From the point of view of special functions, the density of $\X_{a,b,c,d}$ can be expressed by the three-parametric Mittag-Leffler function. 

The classical Mittag-Leffler function is the entire function 
\begin{equation*}
    E_{\rho}(z) := \sum_{n\geq 0} \frac{z^n}{\Gamma(1+\rho n)}, \quad z \in \mathbb{C},\, \rho >0.
\end{equation*}
It was introduced by Gosta Mittag-Leffler in 1903. Wiman introduced the two-parametric Mittag-Leffler function in 1905 and defined it as
\begin{equation*}
    E_{\rho, \mu}(z) := \sum_{n\geq 0} \frac{z^n}{\Gamma(\mu +\rho n)}, \quad z \in \mathbb{C},\, \rho, \mu >0.
\end{equation*}
Let $\rho >0$, $\mu > 0$ and $\gamma > 0$, the three-parametric Mittag-Leffler function is defined as 
\begin{equation*}
    E_{\rho, \mu}^\gamma (z) := \sum_{n\geq 0}   \frac{\Gamma(\gamma + n)}{\Gamma(\gamma)\Gamma(1 + n)\Gamma(\mu +\rho n)} z^n, \quad z \in \mathbb{C}.
\end{equation*}
It can also be represented via the Mellin-Barnes integral:
\begin{equation}\label{eq three-parametric ML}
    E_{\rho,\mu}^\gamma (z) = \frac{1}{2\pi i}\int_L \frac{\Gamma(s)\Gamma(\gamma-s)}{\Gamma(\gamma)\Gamma(\mu-\rho s)} (-z)^{-s}ds,
\end{equation}
where $|\text{arg} \, z|< \pi$, the contour of integration begins at $c-i\infty$, ends at $c+i\infty$, $0 < c < \gamma$, and separates all poles of the integrand at $s = -k,  k = 0, 1, 2, \cdots$ to the left, and all poles at $s = n +\gamma, n=0,1, \cdots$ to the right. Applying the Mellin inversion formula to \eqref{eq three-parametric ML} we obtain, for $ 0 < s < \gamma $,
\begin{equation}\label{eq mellin three-parametric ML}
    \int_0^\infty t^{s-1} E_{\rho,\mu}^\gamma (-t)dt = \frac{\Gamma(s)\Gamma(\gamma-s)}{\Gamma(\gamma)\Gamma(\mu-\rho s)}.
\end{equation}
Therefore, the density of $\X_{a,b,c,d}^{-1}$ ($a,b,c,d >0$) is proportional to $t^{b-1}E_{d, c + bd}^{a+b}(-t)$, see \eqref{eq mellin X inverse} for details.
We then deduce the non-negativity of the three-parametric Mittag-Leffler function in some cases. 
\begin{corollary}\label{coro nonnegativity of three-parametric ML}
Let $ \rho > 0, \mu > 0$, and $\gamma > 0$, $f$ is the function defined in Theorem \ref{Theorem admissible parameter domain}.
\begin{itemize}
    \item[(I)] $ E_{\rho,\mu}^\gamma (z) $ takes negative values on the negative half line if any of the following conditions holds:
    \begin{itemize}
        \item[(1)]  $\rho > 2$;
        \item[(2)]   $\mu < \gamma \rho$;
        \item[(3)]  $\rho = 2, \;\mu < 3 \gamma$;
        \item[(4)] $1 < \rho \leq 2, \; \gamma \geq 1,\; \mu < \gamma \rho + f(\rho)-\rho$.
    \end{itemize}
   
    \item[(II)]  $ E_{\rho,\mu}^\gamma (z) $ is non-negative on the real line if any of the following conditions holds: 
    \begin{itemize}
        \item[(1)]  $\rho \leq 1, \mu \geq \rho \gamma$;
        \item[(2)]  $1 < \rho \leq 2, \; \gamma \leq 1, \; \mu \geq \gamma \rho + f(\rho)-\rho $;
        \item[(3)] $1 < \rho \leq 2, \; 2\mu \geq 3\gamma \rho, \;  2(\frac{\mu}{\rho}-\gamma)(\frac{\mu}{\rho}-\gamma +\frac{1}{2}) \geq  \gamma$.
    \end{itemize}
\end{itemize}
\end{corollary}
Recall that a non-negative function $f$ on $(0, \infty)$ is called completely
monotonic if it has derivatives of all orders and
$ (-1)^n f^{(n)}(x) \geq 0$
 for $n \geq 1$
and $x > 0$.  The complete monotonicity of the three-parametric Mittag-Leffler function has been well-studied, see e.g. \cite{GHLP21,Sib23}, it is proved that if $0 < \rho \leq 1$ and $0 < \gamma \rho \leq \mu$, then $ E_{\rho,\mu}^\gamma (- z) $ is completely
monotonic. However, we did not find reference on the non-negativity of the three-parametric Mittag-Leffler function.


\subsection{Powers of half \texorpdfstring{$\a$}{}-Cauchy variable}
\label{sec alpha cauchy}
Finally, we investigate the infinite divisibility of the powers of half  $\a$-Cauchy variable. The main tools are Theorem \ref{thm main} and  the Kristiansen's Theorem:
\begin{theorem}[Kristiansen \cite{Kri94}]
    \label{Kristiansen's Theorem}
    The independent product $\W \times \G_2 $ is ID for any non-negative random variable $\W$.
\end{theorem}
As a generalization of 
the standard Cauchy distribution $\frac{1}{\pi} \frac{1}{1 + x^2}$, the $\alpha$-Cauchy distribution, denoted by $\mathcal{C}_\a$, with density
\begin{equation}\label{eq density alpha Cauchy}
    f_{\mathcal{C}_\a}(t) = \frac{\sin(\pi/\a)}{2\pi/\a}\frac{1}{1+|t|^\a}, \quad \alpha > 1, \,\, t \in \mathbb{R}, 
\end{equation}
was introduced by
Yano, Yano and Yor \cite{YYY09}. 
 They proposed three open questions (\cite[Remark 2.9]{YYY09}):
\begin{enumerate}
    \item Is $\mathcal{C}_\a$ SD (or ID at least)?
    \item  Is $|\mathcal{C}_\a|$  SD (or ID at least)?
    \item Is $|\mathcal{C}_\a|^{-p}$ SD (or ID at least) for $p > 0$?
\end{enumerate}
 There are very few results about these questions. Bondesson \cite{Bon87} proved that $ |\mathcal{C}_\a|$ is ID for $1 < \a \leq 2$. Di\'edhiou \cite{Die98} proved that $ |\mathcal{C}_2|$ is SD. In the present paper, we establish the infinite divisibility of the powers of $|\mathcal{C}_\a|$ in many cases, which extend Bondesson's theorem.

\begin{theorem}\label{Theorem ID of half alpha Cauchy}
The random variable $|\mathcal{C}_\a |^{ \varepsilon p}, \varepsilon = \pm 1,$ is ID if 
\begin{equation}\label{eq range p}
p \geq 
    \begin{cases}
   (\a + 1)/3 ,& \quad \varepsilon = 1, \, 1 < \a \leq 2. \\
     \a/2 ,& \quad \varepsilon = 1, \, \a > 2. \\
     \a/2 ,& \quad \varepsilon = - 1, \, 1 < \a \leq 2. \\
     (2\a - 1)/3 ,& \quad \varepsilon = -1, \, \a > 2. \\
    \end{cases}
\end{equation}
\end{theorem}

The rest of this paper is organized as follows. In Section \ref{Section previopus results}, we review some previous results on the existence of random variables having moments of Gamma type. In section \ref{section moment problem}, we discuss the connection between moments of Gamma type and Hausdorff moment problem. In Section \ref{Section proof}, we prove our results.

\section{Previous results on moments of Gamma type}
\label{Section previopus results}
\subsection{Some special cases}
The problem of the existence of $\X$ has been investigated in some special cases.  
\begin{enumerate}
    \item Młotkowski and Penson \cite{MP14} proved, for $\a \in (0,1)$ and real $r$, the random variable $\Y_{\a,r}$ with fractional moments 
\begin{equation}
\label{eq mellin Y alpha r}
    \mathbb{E}[\Y_{\a,r}^s] = \frac{\Gamma(r+1+s)}{\Gamma(1+\a s)\Gamma(r+1+(1-\a)s)}  , \quad s > -r.
\end{equation}
exists if and only if $r \in (-1, -1+1/\a]$. 
\item Ferreira and Simon \cite{FS23} proved, for real $\a$ and $\beta$, the random variable $\M_{\a,\beta}$ with fractional moments 
\begin{equation}
\label{eq mellin M alpha beta}
    \mathbb{E}[\M_{\a,\beta}^s] = \Gamma(\a+\beta)\frac{\Gamma(1+s)}{\Gamma(\a+\beta+\a s)}, \quad s > -1,
\end{equation}
exists if and only if $\a \in [0,1], \beta \geq 0$. For $\a = 1, \beta > 0, \M_{1,\beta} \elaw \B_{1, \beta}$. For $\a \in [0, 1), \beta \geq 0$, the density of $\M_{\a,\beta}$ is $\Gamma(\a +\beta) \phi(-\a, \beta, -x)$, where 
$$ \phi(\a, \beta, z) = \sum_{n\geq 0} \frac{z^n}{n! \Gamma(\beta+\a n) }, \quad \beta, z \in \mathbb{C}, \,\, \a > -1, $$
is the Wright function. The unimodality and the log-concavity of the Wright function have been studied in \cite{FS23}. For later use, we introduce the random variable $\M_{\a,\beta, t}$, $t > -1, \a \in [0,1), \beta \geq 0$, whose density is $\frac{\Gamma(\a(1+t)+\beta)}{\Gamma(1+t)} x^{t} \phi(-\a, \beta, -x)$. By direct computation and \eqref{eq mellin M alpha beta}, its fractional moments is 
\begin{equation}
\label{eq mellin M alpha beta t}
    \mathbb{E}[\M_{\a,\beta, t}^s] = \frac{\Gamma(\a(1+t)+\beta)}{\Gamma(1+t)}\frac{\Gamma(1+t+s)}{\Gamma(\a(1+t)+\beta+\a s)}, \quad s > -1-t. 
\end{equation}

\item Dufresne \cite[Theorem 1]{Duf10a} proved, for positive $a,c$ and real $b,d$, the random variable $B^{(a,b,c,d)}$ with fractional moments 
\begin{equation}
\label{eq B abcd}
    \mathbb{E}[(B^{(a,b,c,d)})^s] = \frac{\Gamma(a+b)\Gamma(c+d)}{\Gamma(a)\Gamma(a)}\frac{\Gamma(a+s)\Gamma(c+s)}{\Gamma(a+b+s)\Gamma(c+d+s)}, \quad s > -\min (a,c),
\end{equation}
exists if and only if $b +d \geq 0$ and $\min (a,c) \leq \min (a+b, c+d)$.  In the case $b + d = 0$, $B^{(a,b,c,d)}$ has a point mass $$\frac{\Gamma(a+b)\Gamma(c+d)}{\Gamma(a)\Gamma(c)}$$ at one, see Dufresne \cite[Theorem 6.2]{Duf10b}. 
\end{enumerate}

There are very few results when $A_j A_{j'} < 0$ for some $j, j'$. 

\begin{enumerate}
    \item[(4)] Kadankova, Simon and Wang \cite{KSW20} studied the special case $J = K = 2, A_1 = -A_2  = B_1 = B_2 = 1$. Their Theorems 1 and 2 can be reformulated as the following theorem.
\begin{theorem}\label{Theorem KSW20 THM1}\cite{KSW20}
For every $a, b, c, d > 0$, the distribution 
$$ D\left[ \begin{array}{cc}
    a & b \\
    (c,d) & -
\end{array}  \right]  \elaw D\left[ \begin{array}{cc}
    a & b \\
    (d,c) & -
\end{array}  \right] $$
satisfying
\begin{equation}\label{eq Mellin D}
    \mathbb{E} \left(  D\left[ \begin{array}{cc}
    a & b \\
    (c,d) & -
\end{array}  \right] ^s \right) = \frac{\Gamma(c)\Gamma(d)}{\Gamma(a)\Gamma(b)} \frac{\Gamma(a+s)\Gamma(b-s)}{\Gamma(c+s)\Gamma(d+s)}
\end{equation}

\begin{itemize}
    \item[(a)] exists if $c > a, d > a, c + d \geq 3a + b +\frac{1}{2}$ and $2(c-a)(d-a) \geq a+b$. 
    \item[(b)] does not exist if $c + d < 3a + b +\frac{1}{2}$ or $\min (c,d) \leq a$. 
\end{itemize}
\end{theorem}

\item[(5)] Bosch \cite{Bos15} proved, for $0 < \a < 1$ and $t > 0$, the random variable $ \F_{\a, t} $ with fractional moments
\begin{equation}
   \mathbb{E}[\F_{\a,t}^s] =\Gamma(t)\frac{\Gamma(1- s/t)\Gamma(1+s/t)}{\Gamma(1- \a s/t)\Gamma(1+\a s/t)\Gamma(t+s)}
\end{equation}
exists if $0 < \a \leq 1/2$ and $0 \leq t \leq 1-\a$. Bosch \cite[Proposition 4.1]{Bos15} also showed $\F_{\a,1}$ does not exist. Indeed, $\mathbb{E}[\F_{\a,1}^s]$ vanishes at $-1/\a$, which is impossible if $\F_{\a,1}$ exists. By the following identity in law
\begin{equation}
    \F_{\a,t}^{t/T} \elaw \F_{\a,T} \times \M_{t/T, 0, T-1}, \quad T > t > 0, \, 0 < \a < 1, 
\end{equation}
the existence of $\F_{\a,T}$ implies the existence of $\F_{\a,t}$ for all $t < T$, and the non-existence  of $\F_{\a,t}$ implies the non-existence of $\F_{\a,T}$ for all $T > t$. 
(Throughout, in any factorization of the type $X \elaw Y \times Z$, 
the random variables $Y, Z$ on the right-hand side will be assumed to be independent.) 
Therefore, for $0 < \a \leq 1/2$ there exists $t_\a \in [1-\a, 1)$ such that $ \F_{\a, t} $ exists if and only if $0 < t \leq t_\a$. But the exact value of $t_\a$ is still unknown. By our Theorem \ref{thm main}, we have $t_{1/2} = 1/2$. Indeed, 
\begin{equation}
   \mathbb{E}[\F_{1/2,t}^s] = \frac{\Gamma(t)}{\Gamma(1/2)\Gamma(1/2)}\frac{\Gamma(\frac{1}{2} - \frac{s}{2t})\Gamma(\frac{1}{2}+ \frac{s}{2t})}{\Gamma(t+s)}
\end{equation}
therefore, 
\begin{equation*}
    \F_{1/2,t} \elaw \X_{ 1/2,1/2, t, 2t }^{\frac{1}{2t}}.
\end{equation*}
\end{enumerate}

\subsection{Quasi-infinite divisibility}

It is obvious that $\X$ exists if $\Y := \log \X$ is infinitely divisible (ID), we refer to Bondesson \cite{Bon92}, Sato \cite{Sat99}, Steutel and Van Harn \cite{SH04} for various accounts on infinite divisibility, self-decomposability (SD), generalized Gamma convolutions (GGC) and hyperbolically completely monotone (HCM) densities. But the reverse direction does not hold, namely
\begin{equation}
    \label{non equivalence}
    \text{$\X$ exists  $\, \nRightarrow \,$ $\Y$ is ID}.
\end{equation}
For example (\cite[Remark (c), page 5]{KSW20}), if $\X = D\left[ \begin{array}{cc}
    a & b \\
    (2a+b,a + 1/2) & -
\end{array}  \right]$, see \eqref{eq Mellin D} for definition, then $\Y$ is not ID. If $\X$ is the random variable such that $J= K = 3, \,A_j = B_k = 1,\, a_j = j$, and $b_k$ satisfying 
$$ 1 + \frac{0.015}{1+p} - \frac{0.25}{2+p} + \frac{1}{3+p} = \frac{(b_1+p)(b_2+p)(b_3+p)}{(1+p)(2+p)(3+p)}, $$
$\X$ exists, but $\Y$ is not ID. See 
\cite[Example 2, page 2195]{Duf10b} for more details.

The characteristic function of $\Y$ is given by \eqref{eq def characteristic function of Gamma type}. 
The well-known Malmsten formula (see for instance Carmona, Petit and Yor \cite{CPY97}) for the Gamma function is as follows
\begin{equation}
    \label{eq Malmsten}
    \frac{\Gamma(z + a)}{\Gamma(a)} = \exp \lacc \psi(a)z + \int_0^\infty (e^{-zx} - 1 + zx)\frac{e^{-ax} dx}{x(1-e^{-x})}\racc, \,\, \text{Re}(z) > -\text{Re}(a), \, \text{Re}(a) > 0.
\end{equation}
For $A_j > 0, \, B_k > 0,\, a_j > 0,\, b_k >0$, a repeated use of the Malmsten formula entails that, $\Y$ is ID if and only if,
\begin{equation}
    \label{Condition Y ID}
    \sum_{j=1}^J \frac{e^{-t \, a_j/A_j}}{1- e^{-t/A_j}} -  \sum_{k=1}^K \frac{e^{-t \, b_k/B_k}}{1- e^{-t/B_k}} \geq 0, \quad \forall \, t > 0.
\end{equation}
 However, the condition \eqref{Condition Y ID} is not easy to verify. Karp and Prilepkina \cite[Theorem 5]{KP16} have provided some practical sufficient condition. By \eqref{non equivalence}, the sufficient condition \eqref{Condition Y ID} is not necessary. Shanbhag and Sreehari \cite[Theorem 2]{SS77} proved that $\log \G_a$ is SD for any $a > 0$. If $\X$ exists, then by comparing the Laplace transform we have
 \begin{equation}
     \Y + \sum_{k=1}^K B_k \log \G_{b_k} \elaw \sum_{j=1}^J A_j \log \G_{a_j}. 
 \end{equation}
Using the terminology in Lindner, Pan and Sato \cite{LPS18}, we have the equivalence 
\begin{equation}
    \label{equivalence QID}
    \text{$\X$ exists  $\,\Longleftrightarrow \,$ $\Y$ is quasi-infinite divisible (quasi-ID)}.
\end{equation}
However, characterizing the existence of quasi-ID distributions in a given framework is a hard question. 
It would be interesting to construct more quasi-ID distributions with the help of the Gamma function.

\section{Connections to Hausdorff moment problem}\label{section moment problem}
In this section, we discuss the connection between moments of Gamma type and Hausdorff moment problem. 
Roughly speaking, for a given sequence $(m_0, m_1, \cdots)$, moment problem asks for necessary and sufficient conditions for the existence of a random variable $X$ such that $\mathbb{E}(X^n) = m_n /m_0$, and we call it a moment sequence if such $X$ exists.  
There are three well-known moment problems: the Hamburger 
moment problem deals with all random variables taking values in the whole real line, the Stieltjes moment
problem deals with non-negative random variables, and the Hausdorff moment problem deals with random variables on compact real intervals, which is the one we consider here.
It is worth mentioning that a Stieltjes or a Hamburger moment problem may have infinitely many solutions if it is solvable (indeterminate moment problem) whereas a Hausdorff moment problem always has a unique solution if it is solvable (determinate moment problem). 

The support of $\X$ is $[0,D \prod_{j=0}^J A_j^{A_j} \prod_{k=0}^K B_k^{-B_k}]$ if $A_j > 0$ for all $j$ and $\gamma = 0$, see (23) of \cite{KP16}. Recall that $\gamma = \sum_{j=1}^J |A_j| - \sum_{k=1}^K | B_k|.$ In this case, the problem of the existence of $\X$ is equivalent to the Hausdorff moment problem of the sequence 
\begin{equation}\label{eq def mu_n}
   \mu_n = \frac{ \prod_{j = 1}^{J} \Gamma(A_j n + a_j)}{\prod_{k = 1}^{K}\Gamma(B_k n + b_k)}, \quad n \geq 0. 
\end{equation}
We give a family of sequences of factorial ratios which are Hausdorff moment sequences, and then use it to recover some examples in the literature. 
\begin{proposition}
\label{prop Hausdorff moment sequence}
    If $J = K$, $A_j = B_k = 1$ for all $j, k$, $0 \leq a_1 \leq a_2 \leq \cdots, \; 0 \leq b_1\leq b_2 \leq \cdots$, and $\sum_{j=1}^L a_j \leq \sum_{j=1}^L b_j$ for each $L = 1,2, \cdots, J$, then $ (\mu_n)_{n\geq 0}$ defined by \eqref{eq def mu_n} is a Hausdorff moment sequence. 
\end{proposition}
This proposition is a direct consequence of \cite[Theorem 5 (c)]{KP16}, \eqref{Condition Y ID} and \eqref{equivalence QID}, we omit the proof. 
By Proposition \ref{prop Hausdorff moment sequence} and the Gauss multiplication formula
\begin{equation*}
    \Gamma(mz)(2\pi)^{\frac{m-1}{2}} \, = \, m^{mz-\frac{1}{2}} \Gamma(z)\Gamma(z+\frac{1}{m})\cdots\Gamma(z+\frac{m-1}{m}),
\end{equation*}
we can easily verify that the sequences 
\begin{align*}
        &\mu_n^{(1)} = \frac{ \Gamma(2n+ 1)}{\Gamma(n+1)\Gamma(n +1)}, 
        &\mu_n^{(2)} =  \frac{ \Gamma(6n+1)\Gamma(n+1)}{\Gamma(3n+1)\Gamma(2n+1)\Gamma(2n+1)}, \\
        &\mu_n^{(3)} = \frac{ \Gamma(3n + 1)}{\Gamma(n+1)\Gamma(2n +1)}, 
        &\mu_n^{(4)} =  \frac{ \Gamma(3n + 3)}{\Gamma(n+1)\Gamma(2n +3)}, \\
        &\mu_n^{(5)} = \frac{ \Gamma(6n + 1)\Gamma(4n + 1)}{\Gamma(2n+1)\Gamma(3n +1)\Gamma(5n + 1)},
        &\mu_n^{(6)} =  \frac{ \Gamma(5n + 1)}{\Gamma(2n+1)\Gamma(3n +1)}, \\
        &\mu_n^{(7)} =  \frac{ \Gamma(4n+2)}{\Gamma(n+1)\Gamma(3n+4)}, 
        &\mu_n^{(8)} =  \frac{ \Gamma(4n+4)}{\Gamma(n+1)\Gamma(3n+6)}, 
\end{align*}
are all Hausdorff moment sequence. The sequences $\mu^{(1)}, \mu^{(3)}, \mu^{(4)}, \mu^{(6)}$ are binomial sequences, whose moment problem has been totally solved by Młotkowski and Penson \cite{MP14}. The sequences 
$\mu^{(7)}, \mu^{(8)}$ are the first two terms of the combinatorial numbers of Brown and Tutte, whose moment problem was investigated by Penson et al. \cite{PGHG22}. At the end of \cite{PGHG22}, the authors said "we believe that the sequence $A(M, n)$ belongs to a larger family of
moment sequences for which analogous relations of type $\cdots$", our Proposition \ref{prop Hausdorff moment sequence} gives a kind of such family, but it is not optimal, we leave this question to further research.

\section{Proofs}\label{Section proof}

\subsection{Proof of Theorem \ref{thm main}}
We split Theorem \ref{thm main} into three propositions.

\begin{proposition}\label{proposition main 1}
\begin{itemize}
        \item[(a)] $\X_{a,b,c,d}$ does not exist if $d > 2$.
        \item[(b)] $\X_{a,b,c,d}$ does not exist if $c < ad$.
        \item[(c)]  $\X_{a,b,c,d}$ exists if $0 < d \leq 1$ and $c \geq ad$.
    \end{itemize}
\end{proposition}

\begin{proposition}\label{proposition main 2}
    \begin{itemize}
        \item[(a)]     If $a\in (0,1)$, $\X_{a,1-a,c,d}$ exists if and only if $ d  \leq 2$ and 
    \begin{equation}\label{eq when a+b=1}
    c \geq 
        \begin{cases}
            ad &, \quad  0<d\leq 1 ,\\
            ad -d + f(d)  &,  \quad  1<d\leq 2,
        \end{cases}
    \end{equation}
    where $f$ is the function defined in Theorem \ref{Theorem admissible parameter domain}.
        \item[(b)] $\X_{a,b,c,d}$ exists if $1< d \leq 2, a + b < 1$ and $c$ satisfies \eqref{eq when a+b=1}. 
        \item[(c)]  $\X_{a,b,c,d}$ does not exist if $1 < d \leq 2, a + b > 1$ and $c$ does not satisfy \eqref{eq when a+b=1}.
    \end{itemize}

\end{proposition}

\begin{proposition}\label{proposition main 3}
    \begin{itemize}
        \item[(a)]  $\X_{a,b,c,2}$ does not exist if $3a + b > c$.
        \item[(b)] $\X_{a,b,c,2}$ exists if $c \geq 3a +b$ and $2(\frac{c}{2}-a)(\frac{c}{2}+\frac{1}{2}-a)\geq a+b$.\\
    In particular, if $a+b \geq 1$,  $\X_{a,b,c,2}$ exists if and only if $c \geq 3a + b$.
        \item[(c)] $\X_{a,b,c,d}$ exists if $1< d < 2$,  $\frac{2c}{d} \geq 3a+b$ and $2(\frac{c}{d}-a)(\frac{c}{d}+\frac{1}{2}-a)\geq a+b$. \\
    In particular, if $a+b \geq 1$,  $\X_{a,b,c,d}$ exists if $1< d < 2$ and $\frac{2c}{d} \geq 3a + b$.
    \end{itemize}
\end{proposition}

\subsubsection{Proof of proposition \ref{proposition main 1}}
(a) By Proposition \ref{proposition Jan10 Coro 5.2}, if $\X_{a,b,c,d}$ exists, then it is necessary that $d \leq 2$. \\
(b) By H\"older's inequality, and choosing $p = 1/t,\, q = 1/(1-t)$, we can show that the function $s \, \mapsto \, \mathbb{E}(\X_{a,b,c,d}^s)$ is log-convex on $(-a,b)$. The second derivative of $\log \Gamma(a+s) + \log \Gamma(b-s) - \log \Gamma(c+ds)$ is
\begin{equation}
    \label{eq second derivative}
    \sum_{n\geq 0} \frac{1}{(a+s+n)^2} + \sum_{n\geq 0} \frac{1}{(b-s+n)^2} - \sum_{n\geq 0} \frac{d^2}{(c+ds+n)^2}.
\end{equation}
If $c < ad$, then \eqref{eq second derivative} becomes negative in the neighborhood of $- c/d$, which contradicts the log-convexity, therefore $\X_{a,b,c,d}$ does not exist. \\
(c)  $\X_{a,b,c,d}$ exists because
\begin{equation*}
    \X_{a,b,c,d} \elaw \M_{d, c-ad, a-1} \times \G_b^{-1}, \quad d < 1,
\end{equation*}
recall that the fractional moments of $\M_{d, c-ad, a-1}$ is given by \eqref{eq mellin M alpha beta t} and
\begin{equation*}
    \X_{a,b,c,d} \elaw \B_{a,c-a} \times \G_b^{-1}, \quad d = 1.
\end{equation*}

\subsubsection{Proof of proposition \ref{proposition main 2}} This proof relies on the distribution of zeros of the two-parametric Mittag-Leffler functions. 
From \eqref{eq mellin three-parametric ML}, we have, for $ -b < s < a$,
\begin{equation}
\label{eq mellin X inverse}
    \int_0^\infty t^{s} t^{b-1}E_{d, c + bd}^{a+b}(-t)dt = \frac{\Gamma(b+s)\Gamma(a-s)}{\Gamma(a+b)\Gamma(c - d s)}.
\end{equation}
Comparing with \eqref{def Xabcd}, we know that the density of $\X_{a,b,c,d}^{-1}$ ($a,b,c,d >0$) is proportional to $t^{b-1}E_{d, c + bd}^{a+b}(-t)$. Therefore, 
\begin{equation}\label{fact equiv}
    \text{the existence of $\X_{a,b,c,d}$ is equivalent to the non-negativity of $E_{d, c + bd}^{a+b}(-t)$.}
\end{equation}
 When $a + b = 1$, we recover the two-parametric Mittag-Leffler function,  of which the distribution of the zeros has been studied in numerous works (see e.g. \cite{PS03, Sed04, Psk05, PS13} and the references therein). We collect related known results and reformulate them in the following Theorem \ref{Theorem admissible parameter domain}. More precisely, 
Pskhu \cite[Theorem 2]{Psk05} has proved that $\lacc 0 < \rho \leq 1, \mu \geq \rho  \racc \cup \lacc 1 < \rho \leq 2, \mu \geq 3 \rho /2  \racc  \subseteq \mathcal{D}$ and $\lacc \rho > 0 , \mu < \rho  \racc \cup \lacc  \rho \geq 2, \mu \leq 3 \rho /2, (\rho, \mu) \neq (2,3)  \racc  \subseteq \mathcal{D}^c$. Popov and Sedletskii \cite[Chapter 6]{PS13} thoroughly studied the case $\rho \in (1,2)$, and the functions $L, U$ comes from \cite[Theorem 6.1.3]{PS13}. And \cite[Theorem 2.1.4]{PS13} says that $E_{\rho, \mu}$ takes negative values for any pair in $\lacc \rho > 2, \mu > 0 \racc$. 

\begin{theorem}\label{Theorem admissible parameter domain}\cite{PS13}
Let 
\begin{equation}
    \mathcal{D} := \lacc (\rho, \mu): E_{\rho, \mu}(- t) \geq 0, \, \forall \, t>0 \racc.
\end{equation}
Then there exists an increasing function $f$ on $[1,2]$, such that
\begin{equation}
    \lacc 0 < \rho \leq 1, \mu \geq \rho  \racc \cup \lacc 1 < \rho \leq 2, \mu \geq f(\rho) \racc   = \mathcal{D}, 
\end{equation}  
    where $f(1) = 1$, $f(2) = 3$, and
    \begin{equation}\label{eq def f}
        \text{ $ \rho < L(\rho) < f(\rho) < U(\rho) < 3\rho/2$ for $1 < \rho < 2$, }
    \end{equation} with
    \begin{equation}\label{eq function L}
        L(\rho) =   \begin{cases}
            \rho + \exp (-\pi \cot(\pi(1-1/\rho))) & , \quad 1 < \rho < 3/2, \\
            3(\rho - 1) + 0.7(2-\rho)^2 & , \quad 3/2 \leq \rho < 2,
        \end{cases}
    \end{equation}and 
        \begin{equation}\label{eq function U}
        U(\rho) =   \begin{cases}
            4\rho/3 & , \quad 1 < \rho < 3/2, \\
            2\rho -1 & , \quad 3/2 \leq \rho < 2.
        \end{cases}
    \end{equation}
    
\end{theorem}

Then Proposition \ref{proposition main 2} (a) follows from \eqref{fact equiv} and Theorem \ref{Theorem admissible parameter domain}. Proposition \ref{proposition main 2} (b) follows from Proposition \ref{proposition main 2} (a) and the following identity in law: 
\begin{equation}
     \X_{a,b,c,d} \elaw  \X_{a,1-a,c,d} \times \B_{b,1-a-b}^{-1}, 
\end{equation}
where $d \leq 2, a + b < 1$ and $c$ satisfies \eqref{eq when a+b=1}. 

We suppose that there exists a random variable $\X_{a,b,c,d}$, where $d \leq 2, a + b > 1$ and $c$ does not satisfy \eqref{eq when a+b=1}. Then from
\begin{equation}
     \X_{a,1-a,c,d} \elaw  \X_{a,b,c,d} \times \B_{1-a,a+b-1}^{-1}, 
\end{equation}
$ \X_{a,1-a,c,d}$ exists, which contradicts Proposition \ref{proposition main 2} (a), hence, we have proved Proposition \ref{proposition main 2} (c).

\subsubsection{Proof of proposition \ref{proposition main 3}} This proof relies on the result of Kadankova, Simon and Wang \cite{KSW20}. 
When $d = 2$, by the Legendre duplication formula for the gamma function
\begin{equation*}
    \Gamma(1/2)\Gamma(2x) = 2^{2x-1}\Gamma(x)\Gamma(x+1/2),
\end{equation*}
we have that $\X_{a,b,c,d}$ is equal in law to $D\left[ \begin{array}{cc}
    a & b \\
    (\frac{c}{2},\frac{c+1}{2}) & -
\end{array}  \right] $, up to a multiplicative constant. Then Proposition \ref{proposition main 3} (a) and (b) follows from Theorem \ref{Theorem KSW20 THM1}. \\
Proposition \ref{proposition main 3} (c) follows from Proposition \ref{proposition main 3} (b) and the identity in law
\begin{equation}
    \X_{a,b,c,d} \elaw \X_{a,b,\frac{2c}{d},2} \times \M_{\frac{d}{2},0,\frac{2c}{d}-1}^{2}, \quad d < 2, 
\end{equation}
recall that the fractional moments of $\M_{\frac{d}{2},0,\frac{2c}{d}-1}$ is given by \eqref{eq mellin M alpha beta t}.

\subsection{Proof of Theorems \ref{Theorem ID of half alpha Cauchy}}
This proof relies on Theorem \ref{thm main} and Kristiansen's Theorem. 
First, we observe that 
\begin{equation}\label{eq Mellin alpha Cauchy}
    \mathbb{E}\left[|\mathcal{C}_\a|^s\right] =  \frac{\sin(\pi/\a)}{\pi} \Gamma\left(\frac{1}{\a} +\frac{s}{\a}\right)\Gamma\left(1 - \frac{1}{\a} - \frac{s}{\a}\right), \quad -1 < s < \a-1.
\end{equation}
Therefore, we have
\begin{equation}
    |\mathcal{C}_\a| \elaw \G_{1/\a}^{1/\a} \times  \G_{1-1/\a}^{-1/\a}.
\end{equation}
For $|p| \geq 1$, it is easy to see that $\G_a^p$ is HCM. 
Hence, 
\begin{equation}\label{result for |p| >= alpha}
    \text{$|\mathcal{C}_\a|^p$ is HCM if $|p| \geq \a$. }
\end{equation}
The right hand side of \eqref{eq Mellin alpha Cauchy} tends to $1/(1+s)$ as $\a$ tends to infinity, thus 
\begin{equation}
    |\mathcal{C}_\a| \claw \mathcal{U}, \qquad \text{as} \, \,\,\a \rightarrow \infty, 
\end{equation}
where $\mathcal{U}$ is a uniform random variable on $(0,1)$, which is not ID. Therefore, $|\mathcal{C}_\a|$ is not ID when $\a$ is large enough. On the other hand, $ \mathcal{U}^{-p}$ is GGC since $\mathcal{U}^{-p} - 1$ is HCM, for any $p > 0$.

\begin{proof}[Proof of Theorem \ref{Theorem ID of half alpha Cauchy}]
By Proposition \ref{proposition Jan10 Coro 5.2}, for $p < \a/2$, $|\mathcal{C}_\a |^{\varepsilon p}$ is not a Gamma mixture, the Kristiansen's Theorem does not work in this case. For $p \geq \a$, we have already seen that $|\mathcal{C}_\a |^{\varepsilon p}$ is HCM, hence ID. We only consider the case $p \in [\a /2, \a)$. 
For $ 1 \leq q < 2$, 
\begin{equation}
    |\mathcal{C}_\a |^{q\frac{\a}{2}} \elaw \X_{\frac{1}{\a},1-\frac{1}{\a},\mu - \frac{2(\a-1)}{q\a},\frac{2}{q}}^{\frac{q}{2}} \times \G_{\mu - \frac{2(\a-1)}{q\a}}, \quad \mu \geq f(2/q),
\end{equation}

\begin{equation}
    |\mathcal{C}_\a |^{-q\frac{\a}{2}} \elaw \X_{1-\frac{1}{\a},\frac{1}{\a},\mu - \frac{2}{q\a},\frac{2}{q}}^{\frac{q}{2}} \times \G_{\mu - \frac{2}{q\a}}, \quad \mu \geq f(2/q). 
\end{equation}

$|\mathcal{C}_\a |^{ \varepsilon q\frac{\a}{2}}, \varepsilon = \pm 1,$ can be a $\G_2$-mixture if and only if 
\begin{equation} 2 \geq 
    \begin{cases}
        f(2/q) - \frac{2(\a-1)}{q\a}  &, \quad \varepsilon = 1. \\
        f(2/q) - \frac{2}{q\a}  &, \quad \varepsilon = -1.  \\
    \end{cases}
\end{equation}
where $f$ is defined in \eqref{eq def f}. 
By Kristiansen's Theorem, we have that 
$|\mathcal{C}_\a |^{ \varepsilon q\frac{\a}{2}}, \varepsilon = \pm 1,$ is ID if 
\begin{equation}  f(2/q) \leq 
    \begin{cases}
       2 + \frac{2(\a-1)}{q\a}  &, \quad \varepsilon = 1. \\
       2 + \frac{2}{q\a}  &, \quad \varepsilon = -1.  \\
    \end{cases}
\end{equation}
Recall that  $f(2/q) < U(2/q)$, $U$ is defined in \eqref{eq function U}, we then have 
$|\mathcal{C}_\a |^{ \varepsilon q\frac{\a}{2}}, \varepsilon = \pm 1,$ is ID if 
\begin{equation}  q \geq 
    \begin{cases}
     \max (1, 2(\a+1)/(3\a)  &, \quad \varepsilon = 1. \\
      \max (1, (4\a -2)/(3\a))  &, \quad \varepsilon = -1.  \\
    \end{cases}
\end{equation}
Let $p = \frac{\a}{2}q$, we can obtain \eqref{eq range p}. 
\end{proof}

\bigskip



\bibliographystyle{plain} 
\bibliography{MomentGammaML}
\end{document}